\documentclass[fleqn]{article}
\usepackage{fleqn}
\usepackage{graphicx}
\usepackage{epsfig}
\usepackage{amsmath, amsthm}
\usepackage{amssymb}
\usepackage{calc}

\newtheorem{teo}{Theorem}[section]

\newtheorem{lem}[teo]{Lemma}
\newtheorem{prop}[teo]{Proposition}
\newtheorem{defi}[teo]{Definition}

\newtheorem{rem}[teo]{Remark}

\newcommand{\beq}{\begin{equation}}
\newcommand{\eeq}{\end{equation}}
\newcommand{\bea}{\begin{eqnarray}}
\newcommand{\eea}{\end{eqnarray}}

\renewcommand{\to}{\ensuremath{\longrightarrow}}
\newcommand{\cP}{\ensuremath{\mathcal{P}}}
\newcommand{\cL}{\ensuremath{\mathcal{L}}}
\newcommand{\cQ}{\ensuremath{\mathcal{Q}}}

\newcommand{\cU}{\ensuremath{\mathcal{U}}}
\newcommand{\cV}{\ensuremath{\mathcal{V}}}
\newcommand{\cD}{\ensuremath{\mathcal{D}}}
\newcommand{\cS}{\ensuremath{\mathcal{S}}}

\newcommand{\M}{{\mathrm{M}}}
\newcommand{\bS}{{\mathbf{S}}}

\newcommand{\bD}{{\mathbf{D}}}
\newcommand{\bU}{{\mathbf{U}}}\newcommand{\bV}{{\mathbf{V}}}
\newcommand{\s}{{\mathrm{S}}}
\newcommand{\1}{\ensuremath{{\bf 1}}}

\DeclareMathOperator{\rk}{rank}

\DeclareMathOperator{\PG}{PG}
\DeclareMathOperator{\GF}{GF}
\DeclareMathOperator{\GL}{GL}
\DeclareMathOperator{\GQ}{GQ}

\begin{document}
\Large
\begin{center}
{\bf Invertible Symmetric 3 $\times$ 3 Binary Matrices\\ and GQ(2,\,4)}
\end{center}
\large \vspace*{-.1cm}
\begin{center}
Andrea Blunck,$^{1}$ P\'eter L\'evay,$^{2}$ Metod Saniga$^{3}$ and P\'eter
Vrana$^{2}$
\end{center}
\vspace*{-.4cm} \normalsize
\begin{center}
$^{1}$Department of Mathematics, University of Hamburg, D-20146 Hamburg,
Germany

\vspace*{.0cm} $^{2}$Department of Theoretical Physics, Institute of Physics,
Budapest University of\\ Technology and Economics, H-1521 Budapest, Hungary

\vspace*{.0cm} and

\vspace*{.0cm}

$^{3}$Astronomical Institute, Slovak Academy of Sciences\\
SK-05960 Tatransk\' a Lomnica, Slovak Republic

\vspace*{.2cm} (\today)

\end{center}

\vspace*{-.3cm} \noindent \hrulefill

\vspace*{.1cm} \noindent {\bf Abstract:}
We reveal an intriguing connection between the set of 27 (disregarding the identity) invertible symmetric $3 \times 3$  matrices over GF(2) and the points of the generalized quadrangle GQ$(2,4)$. The 15 matrices with eigenvalue one correspond to a copy of the subquadrangle GQ$(2,2)$, whereas the 12 matrices without eigenvalues have their geometric counterpart in the associated double-six. The fine details of this correspondence, including the precise algebraic meaning/analogue of collinearity, are furnished by employing the representation of GQ$(2,4)$ as a quadric in PG$(5,2)$ of projective index one. An interesting physical application of our findings is also mentioned.

\vspace*{.3cm}

\vspace*{-.2cm} \noindent \hrulefill

\section{Introduction}

The set of invertible symmetric $3\times 3$ matrices over the field $\GF(2)$ has $28$ elements. The elements different from the identity matrix can be divided in a natural way  into one set of $15$ matrices and two sets of $6$ matrices each. This suggests that there might be a connection to the generalized quadrangle $\GQ(2,4)$, which has a description using a subquadrangle of $15$ points and two additional sets of $6$ points each.

In this paper we show that indeed the set of $27$ invertible, non-identity matrices can be equipped with the structure of $\GQ(2,4)$. In Section 2, we recall some well known facts on generalized quadrangles, in particular, the two descriptions of $\GQ(2,4)$ we are going to work with. In Section 3 we study the set of invertible symmetric $3\times 3$ matrices over $\GF(2)$. In Section 4 we show that the $27$ matrices can be mapped bijectively onto a quadric of projective index one in $\PG(5,2)$, which can be seen as the point set of $\GQ(2,4)$. In Section 5, we give another description, where the points appear as planes in $\PG(5,2)$; in this section we also state an explicit isomorphism between the generalized quadrangle on the $27$ matrices and the model of $\GQ(2,4)$ mentioned above (Proposition \ref{iso}). Finally, in Section 6 we state an application to physics.

The geometry of matrices defined over fields/division rings represents an important branch of geometry (see, e.\,g., the monograph of Wan \cite{wan}). The elements of a matrix space are viewed as vertices of a graph where a particular symmetric and anti-reflexive adjacency relation is defined. The problem of describing all adjacency-preserving bijections between matrix spaces is then converted into that of finding all isomorphisms between the corresponding graphs. A prominent place in the field is occupied by symmetric matrices (e.\,g., \cite{wan1,wan2}), and in particular those defined over fields of characteristic two (e.\,g., \cite{gao,wan3}).  This paper deals, like \cite{wan3}, with a particular case of $3 \times 3$ symmetric binary matrices. Yet our approach is fundamentally different from that followed by Wan and/or others since we focus exclusively on invertible matrices and instead of the concept of adjacency in a graph we use that of collinearity in a particular well-known finite geometry (here the generalized quadrangle of type GQ$(2,4)$).

\section{The generalized quadrangle GQ(2,\,4)}

A finite \emph{generalized quadrangle} is a point-line incidence geometry
$(\cP,\cL)$ satisfying the following axioms (see \cite[1.1]{pt}), where $s,t\ge
1$ are fixed natural numbers:
\begin{itemize}
\item Each point is incident with $1+t$ lines and any two points are joined
    by at most one line.
\item Each line is incident with $1+s$ points and any two lines meet in at
    most one point.
\item For each pair $(p,L)\in\cP\times\cL$ with $p\notin L$ there is a
    unique pair $(q,M)\in\cP\times\cL$ such that $p\in M$, $q\in M$, $q\in
    L$.
\end{itemize}
The pair $(s,t)$ is called the \emph{order} of $(\cP,\cL)$.
Examples of generalized quadrangles (GQs) can be obtained from quadratic forms
(see \cite[3.1.1]{pt}): Let $F$ be a finite field of order $s$ and let $q$ be a
quadractic form of index $2$ on $V=F^{d+1}$. This means that the associated
quadric $\cQ=\{Fv\mid v\in V, v\ne 0, q(v)=0\}$ in the $d$-dimensional
projective space $\PG(d,s)$ over $F$ contains lines but no planes ($\cQ$ has
\emph{projective index} $1$). Recall that such a quadratic form $q$ exists
exactly if $d=3,4$ or $5$.
Now consider the incidence structure $(\cQ,\cL)$, where $\cL$ is the set of
lines of $\PG(d,s)$ contained in $\cQ$. Then $(\cQ,\cL)$ is a generalized
quadrangle of order $(s,t)$, where $t=1$ if $d=3$, $t=s$ if $d=4$, and $t=s^2$
if $d=5$.
This GQ is called an \emph{orthogonal quadrangle}. Since there is up to
equivalence only one quadratic form of index $2$ on each of the vector spaces
under consideration, a finite orthogonal quadrangle is uniquely determined by
the parameters $d$ and $s$ and thus is called $Q(d,s)$ for short.
Note that one can also define infinite generalized quadrangles and then one
similarly gets orthogonal quadrangles over infinite fields (see \cite{hvm}).

We are going to study the orthogonal quadrangle $Q(5,2)$ whose point set $\cQ$
is a quadric in $\PG(5,2)$. One can show (\cite[5.3.2]{pt}) that there is (up
to isomorphism) only one generalized quadrangle of order $(2,4)$, namely,
$Q(5,2)$. For this reason the generalized quadrangle $Q(5,2)$ is also called
$\GQ(2,4)$. It has $(s+1)(s^3+1)=27$ points and $(s^2+1)(s^3+1)=45$ lines.
Non-degenerate hyperplane sections of $\cQ$ give rise to subgeometries of
$Q(5,2)$ that are ismorphic to $Q(4,2)\cong \GQ(2,2)$, the unique GQ of order
$(2,2)$. There are 36 distinct copies of $\GQ(2,2)$ in $\GQ(2,4)$ \cite{sanigaetal}.

We present another model of $\GQ(2,4)$ (\cite[6.1]{pt}; see Figure 1): We start with a
subgeometry isomorphic to $\GQ(2,2)$ (sometimes called the \emph{doily}). Its point set $\cP'$ consists of the
$2$-element subsets of $\{1,2,3,4,5,6\}$ (so there are $15$ points), and the
line set $\cL'$ consists of the $3$-element subsets $\{A,B,C\}$ of $\cP'$ with
$A\cup B\cup C=\{1,2,3,4,5,6\}$ (so there are also $15$ lines).
Now we add $12$ extra points called $1,2,\ldots,6$, $1',2',\ldots,6'$ and the
$30$ extra lines $\{i,\{i,j\},j'\}$ (with $1\le i,j\le 6$, $i\ne j$). The thus
obtained incidence geometry $(\cP,\cL)$ is isomorphic to $\GQ(2,4)$ and depicted in Figure 1. We
sometimes call the additional points $1,2,\ldots,6$, $1',2',\ldots,6'$ the
\emph{double-six}. Its connection to the classical double-six of lines in
$3$-space is established, e.\,g., in \cite{freud}.

\begin{figure}[h]
\centerline{\includegraphics[width=7.5truecm,clip=]{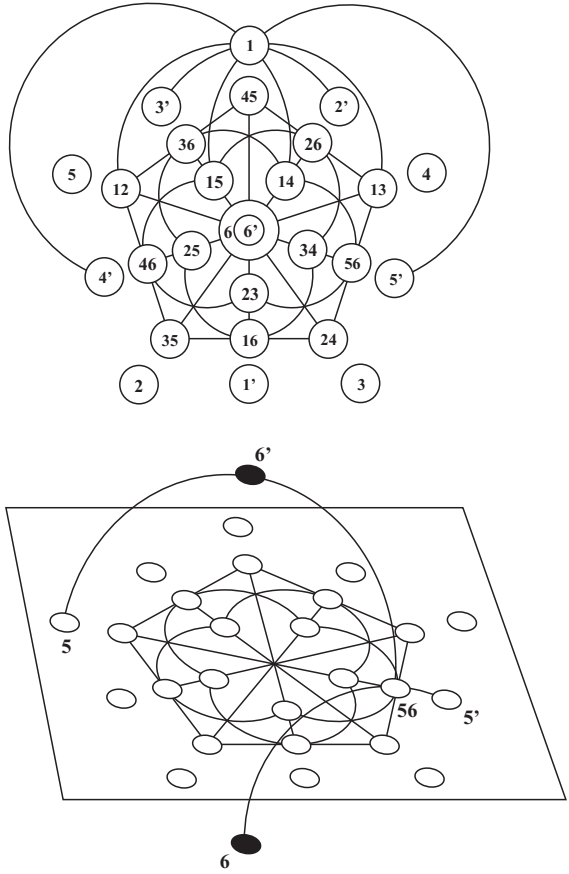}}
\vspace*{0.5cm} \caption{A diagrammatic illustration of the structure of the generalized quadrangle GQ$(2,4)$ built around the doily (after Polster \cite{polster}, slightly modified). In both the parts, each picture depicts the doily and all 27 points (circles) of GQ$(2,4)$. The top picture shows only 19 lines (line segments and arcs of circles) of GQ$(2,4)$, with the two points located in the middle of the doily being regarded as lying one above and the other below the plane the doily is drawn in. Sixteen out of the missing 26 lines can be obtained by successive rotations of the figure through 72 degrees around the center of the pentagon. The bottom picture shows a couple of lines which go off the doily's plane; the remaining 8 lines of this kind are again gotten by rotating the figure through 72 degrees around the center of the pentagon.}
\end{figure}


\section{The 27 matrices}

The set $J=\s(3,2)$ of symmetric $3\times 3$ matrices with entries in $\GF(2)$
is a $6$-dimensional vector subspace of the algebra $\M(3,2)$ of all $3\times
3$ matrices with entries in $\GF(2)$. It is not closed w.r.t.\ multiplication,
so it is not a subalgebra. However, it is closed w.r.t.\ the operations
$A\mapsto A^{-1}$ (if $A\in J^*$, the set of invertible elements of $J$) and
$(A,B)\mapsto ABA$. Since, moreover, the identity matrix $\1$ belongs to $J$,
we have that $J$ is a \emph{Jordan-closed Jordan system} (see \cite{bh}).

Among the $64$ elements of $J=\s(3,2)$, there are $28$ invertible ones. These
fall into three classes:
\begin{itemize}
\item the identity matrix $\1$.
\item $15$ matrices $\ne \1$ with eigenvalue $1$. These are the matrices
    $D_i$ (see below). The matrices $D_1, D_2$ and $D_3$ are involutions
    (i.e.\ $D_i^2=\1$) and have a $2$-dimensional eigenspace, the others
    have a $1$-dimensional eigenspace. We call the set of these matrices
    $\bD$.
\item $12=6+6$ matrices without eigenvalues. These fall into two classes
    $\bU$ and $\bV$ of six matrices each (the matrices $U_i$ and $V_i$,
    respectively). Both subsets $\bU\cup\{\1,0\}$ and $\bV\cup\{\1,0\}$ of $\s(3,2)$, where $0$ is  the zero matrix, are fields isomorphic
    to $\GF(8)$.
\end{itemize}
$$
D_1=\begin{pmatrix}0&0&1\\0&1&0\\1&0&0\end{pmatrix},\quad
D_2=\begin{pmatrix}1&0&0\\0&0&1\\0&1&0\end{pmatrix},\quad
D_3=\begin{pmatrix}0&1&0\\1&0&0\\0&0&1\end{pmatrix},$$
$$
D_4=\begin{pmatrix}0&0&1\\0&1&0\\1&0&1\end{pmatrix},\quad
D_5=\begin{pmatrix}0&1&1\\1&0&1\\1&1&1\end{pmatrix},\quad
D_6=\begin{pmatrix}0&1&1\\1&1&1\\1&1&0\end{pmatrix},$$
$$
D_7=\begin{pmatrix}0&1&0\\1&1&0\\0&0&1\end{pmatrix},\quad
D_8=\begin{pmatrix}1&0&0\\0&0&1\\0&1&1\end{pmatrix},\quad
D_9=\begin{pmatrix}1&0&0\\0&1&1\\0&1&0\end{pmatrix},$$
$$
D_{10}=\begin{pmatrix}1&0&1\\0&1&0\\1&0&0\end{pmatrix},\quad
D_{11}=\begin{pmatrix}1&0&1\\0&1&1\\1&1&1\end{pmatrix},\quad
D_{12}=\begin{pmatrix}1&1&0\\1&0&0\\0&0&1\end{pmatrix},
$$
$$
D_{13}=\begin{pmatrix}1&1&0\\1&1&1\\0&1&1\end{pmatrix},\quad
D_{14}=\begin{pmatrix}1&1&1\\1&0&1\\1&1&0\end{pmatrix},\quad
D_{15}=\begin{pmatrix}1&1&1\\1&1&0\\1&0&1\end{pmatrix},
$$
$$
U_{1}=\begin{pmatrix}1&1&1\\1&1&0\\1&0&0\end{pmatrix},\quad
U_{2}=\begin{pmatrix}1&0&1\\0&0&1\\1&1&1\end{pmatrix},\quad
U_{3}=\begin{pmatrix}0&1&1\\1&0&0\\1&0&1\end{pmatrix},
$$
$$
U_{4}=\begin{pmatrix}0&0&1\\0&1&1\\1&1&0\end{pmatrix},\quad
U_{5}=\begin{pmatrix}0&1&0\\1&1&1\\0&1&1\end{pmatrix},\quad
U_{6}=\begin{pmatrix}1&1&0\\1&0&1\\0&1&0\end{pmatrix}.
$$
$$
V_{1}=\begin{pmatrix}0&1&0\\1&0&1\\0&1&1\end{pmatrix},\quad
V_{2}=\begin{pmatrix}0&1&1\\1&1&0\\1&0&0\end{pmatrix},\quad
V_{3}=\begin{pmatrix}1&1&0\\1&1&1\\0&1&0\end{pmatrix},
$$
$$
V_{4}=\begin{pmatrix}1&1&1\\1&0&0\\1&0&1\end{pmatrix},\quad
V_{5}=\begin{pmatrix}1&0&1\\0&0&1\\1&1&0\end{pmatrix},\quad
V_{6}=\begin{pmatrix}0&0&1\\0&1&1\\1&1&1\end{pmatrix}.
$$
We are interested in the set $\bS:=J^*\setminus\{\1\}=\bD\cup\bU\cup\bV$. Our
aim is to introduce the structure of $\GQ(2,4)$ on $\bS$ such that $\bD$
corresponds to a copy of the doily and $\bU,\bV$ correspond to the associated double-six.

We mention the following geometric facts: All matrices of
$\bS\subseteq\GL(3,2)$ act on the projective plane $\PG(2,2)$ (the Fano plane)
as collineations. The fixed points of such a collineation are exactly the
$1$-dimensional eigenspaces of the associated matrix. This implies that the
collineations induced by $D_1,D_2,D_3$ fix one line pointwise (and hence are
axial collineations), the collineations induced by the other elements of $\bD$
fix exactly one point, and the collineations induced by $\bU\cup\bV$ are fixed
point free. The sets $G=\bU\cup\{\1\}$ and $H=\bV\cup\{\1\}$ are groups
isomorphic to the multiplicative group of $\GF(8)$ (and hence isomorphic to the
cyclic group of order $7$). They induce on $\PG(2,2)$ so-called Singer cycles
(compare \cite[II.10]{HP}).
\section{Two quadratic forms}

The vector space $J=\s(3,2)$ over $F=\GF(2)$ is $6$-dimensional. We identify
$J\setminus\{0\}$ with the projective space $\PG(5,2)$ by identifying the
$1$-dimensional subspace $FX=\{0,X\}$ with $X\in J\setminus\{0\}$. In order to
show that the $27$-element set $\bS=J^*\setminus\{\1\}$ from above can be
viewed as the point set of $\GQ(2,4)$ we are going to construct a suitable
quadratic form on~$J$.

We know that a matrix $X \in J$ belongs to $J^*$ if, and only if, $\det X=1$.
The mapping $\det : J\to F$ is a cubic form. Our first step now is to introduce
new coordinates on $J$ such that $\det$ can be seen as a quadratic form.
For this purpose, we map the matrix
\begin{equation*}
X= \begin{pmatrix}a&b&c\\b&d&e\\c&e&f\end{pmatrix}\in J
\end{equation*}
to the vector
\begin{equation}\label{alpha}\alpha(X)=(a,e^2+df,d, c^2+af, f, b^2+ad)=(a, e+df,d,c+af,f,b+ad)\in V=F^6.
\end{equation}
Note that here we used that $x^2=x$ is valid for all $x\in F$. One can easily
check that $\alpha$ is a bijection (but of course not a linear one).
On $V$ we have the quadratic form $q_0$ of index $3$ given by
\begin{equation*}
q_0(x)=q(x_1,x_2,x_3,x_4,x_5,x_6) = x_1x_2+x_3x_4+x_5x_6.
\end{equation*}
A direct computation, using that $F$ has characteristic $2$, yields that for
each matrix $X\in J$ we have \beq \label{qdet} \det X=q_0(\alpha(X)). \eeq Let
$ \langle .\vert .\rangle_0$ be the symmetric bilinear form associated to
$q_0$, i.\,e., for $x,y\in V$ we have
\begin{equation*}
\langle x\vert y\rangle_0=q_0(x+y)+q_0(x)+q_0(y).
\end{equation*}
Then for all $X,Y\in J$ the following holds:
\beq \label{bilindet}
\langle\alpha(X)\vert\alpha(Y)\rangle_0=\det(X+Y)+\det X+\det Y.
\eeq
Note that
this is not completely trivial because $\alpha$ is not linear.
Equations (\ref{qdet}) and (\ref{bilindet}) allow us to identify $X$ with
$\alpha(X)$ (and so $J$ with $V=F^6$) and $q_0$ with $\det$. In particular, the
identity matrix $\1$ corresponds to the vector $(1,1,1,1,1,1)$.
The quadric $\cQ_0$ given by $q_0$ is the Klein quadric in $\PG(5,2)$, which is
in $1-1$ correspondence with the set of lines of $\PG(3,2)$, and
hence consists of $35$ elements. This is clear also because there are $35$
non-zero non-invertible matrices in $J=\s(3,2)$.

Our set $\bS$, together with the identity matrix $\1$, is the (set-theoretic)
complement of $\cQ_0$ in $\PG(5,2)$. In order to find a bijection from $\bS$ to
a quadric of projective index $1$ in $\PG(5,2)$ we consider another quadratic
form $q$ on $V=J$, given by
\begin{equation*}
q(X)=q_0(X) +(\langle X\vert\1\rangle_0)^2=q_0(X) +\langle X\vert\1\rangle_0.
\end{equation*}
One can easily check that the associated bilinear form $ \langle .\vert
.\rangle$ coincides with $ \langle .\vert .\rangle_0$. Moreover, using
(\ref{qdet}), (\ref{bilindet}) and identifying $X$ with $\alpha(X)$, we get
\beq\label{qdet2}
q(X)=\det X+\det (X+\1)+\det X+\det \1=\det(X+\1)+1.
\eeq
From this equation  it is obvious that the quadric $\cQ$ given by $q$ consists exactly of the points $X$ of $\PG(5,2)$ with $X+\1\in \bS$ (recall that $X= 0 $ is not a point of $\PG(5,2)$).
So it is reasonable to consider the mapping
\begin{equation*} \pi: X\mapsto X+\1,
\end{equation*}
which can be seen as a translation on the affine space $V=J$. However, we want to
interpret $\pi$ projectively. In $\PG(5,2)$, the mapping $\pi$ is not defined
on the point $\1$, because $\pi(\1)=0$; on the other points $\pi$ is a
bijection. For each point $X\ne \1$ the image $\pi(X)$ is the unique third
point on the line joining $X$ and $\1$.

As mentioned above, we have that $\pi(\bS)=\cQ$. So the quadric $\cQ$ consists
of $27$ points. Thus $\cQ$ must be a quadric of projective index $1$ and hence
gives rise to a generalized quadrangle $\GQ(\cQ)\cong\GQ(2,4)$ as described in
Section 2.
This leads us to the following definition:
\begin{defi}
The incidence structure $\GQ(\bS)$ is given by the point set $\bS$ and the line
set $\cL(\bS)$ consisting of the $3$-element subsets $\pi^{-1}(L)$, where
$L\subseteq\cQ$ is a line entirely contained in $\cQ$ (i.e., a line of the
generalized quadrangle $\GQ(\cQ)$ on $\cQ$).
Then clearly $\GQ(\bS)$ is a generalized quadrangle isomorphic to $\GQ(\cQ)$
via $\pi$, i.e., $\GQ(\bS)$ is a model of $\GQ(2,4)$.
For $X,Y\in\bS$ we write $X\sim Y$ iff $X,Y$ are collinear in $\GQ(\bS)$, i.e.,
iff $\pi(X),\pi(Y)$ are collinear in $\GQ(\cQ)$.
\end{defi}
We mention here that similarly to  the quadratic form $q$, which is defined with the help of  the distinguished matrix $\1$, one could define a  quadratic form $q_M$ for each matrix $M\in \bS$ by $q_M(X)=q_0(X) +\langle X\vert M\rangle_0=\det(X+M)+1$. Thus we get $27$ more quadrics of projective index $1$, all of which can be mapped bijectively onto $\bS$ by a mapping $\pi_M:X\mapsto X+M$. However, we will restrict ourselves to $q$ and $\pi$ as above.

The following holds:
\begin{prop}\label{collinear}
Let $X,Y\in\bS$. Then $X\sim Y \iff \langle X+\1\vert Y+\1\rangle=0$.
In particular:
\begin{enumerate}
\item[\rm(a)] If $X,Y\in\bU\cup\bV$, then $X\sim Y \iff \det (X+Y)=0$.
\item[\rm(b)] If $X,Y\in\bD$, then $X\sim Y\iff \det(X+Y)=0$.
\item[\rm(c)] If $X\in\bU\cup\bV, Y\in\bD$, then $X\sim Y\iff \det(X+Y)=1$.
\end{enumerate}
\end{prop}
\begin{proof}
By definition, $X\sim Y$ iff $\pi(X)=X+\1$ and $\pi(Y)=Y+\1$ are collinear in
$\GQ(\cQ)$. This is equivalent to $0=\langle X+\1,Y+\1\rangle=\langle
X+\1,Y+\1\rangle_0$.
Using (\ref{bilindet}), we compute $0=\langle
X+\1,Y+\1\rangle_0=\det(X+\1+Y+\1)+{\det(X+\1)}+\det(Y+\1)=\det(X+Y)+\det(X+\1)+\det(Y+\1)$.
This implies the statements of (a), (b), (c) because in cases (a) and (b) we
have $\det(X+\1)=\det(Y+\1)$ while in case (c) we have
$\det(X+\1)\ne\det(Y+\1)$.
\end{proof}
\begin{prop}
\begin{enumerate}
\item[\rm(a)] For the subsets $\bD$, $\bV$ and $\bU$ of $\bS$, the following
    holds:
\begin{enumerate}
\item[\rm(i)] $\pi(\bU)=\bU$, $\pi(\bV)=\bV$, $\pi(\bD)\cap\bS=\emptyset$,
\item[\rm(ii)] $\bS\cap\cQ=\bU\cup\bV$, $\cQ_0\cap\cQ=\pi(\bD)$,
\end{enumerate}
\item[\rm(b)] Let $\1^\perp(=\1^{\perp_0})$ be the hyperplane of points of
    $\PG(5,2)$ perpendicular to $\1$ w.r.t.\ $\langle .\vert .\rangle(=
    \langle .\vert .\rangle_0)$, i.e., the tangent hyperplane through $\1$
    of $\cQ$ and of $\cQ_0$. Then:
\begin{enumerate}
\item[\rm(i)] $\1^\perp=\{\1\}\cup\bD\cup\pi(\bD)$,
\item[\rm(ii)] The tangents through $\1$ of $\cQ$ and of $\cQ_0$ are the lines
    $\{\1,X,\pi(X)\}$ with $X\in\bD$.
\item[\rm(iii)] The quadric $\cQ\cap\1^\perp=\cQ_0\cap\1^\perp$ equals
    $\pi(\bD)$. So $\pi(\bD)$ is a quadric of projective index $1$ in
    the $4$-dimensional projective space $\1^\perp$ and gives rise to a
    subquadrangle $\GQ(\bD)$ isomorphic to $\GQ(2,2)$.
\end{enumerate}
\end{enumerate}
\end{prop}
\begin{proof}
(a) (i): The statements on $\bU$ and $\bV$ can be checked easily. (They follow
also from the fact that $\bU\cup\{\1,0\}$ and $\bV\cup\{\1,0\}$ are fields of characteristic $2$.)
The statement on $\bD$ is clear because the elements of $\bD$ have eigenvalue
$1$.

(ii): The first equation follows from (i) because $\cQ=\pi(\bS)$. The second
equation follows from the fact that $\bD$ consists exactly of the elements of
$\bS$ that have eigenvalue $1$.

(b) (i) $X\in\1^\perp\iff 0=\langle X,\1\rangle=\langle
X,\1\rangle_0=\det(X+\1)+\det X+1\iff \det (X+\1)\ne \det X$. So
$\1\in\1^\perp$, and for $X\ne\1$ we get $X\in\1^\perp\iff X\in\bD$ or
$X+\1\in\bD$.

(ii) and (iii) follow from (i) and (a).
\end{proof}

\section{Another projective representation}

Again we consider the projective space $\PG(5,2)$. But now we study the planes
of this space, i.\,e., the $3$-dimensional vector subspaces of $V=F^6$. Each such
plane can be described as the row space of a $3\times 6$ matrix, i.\,e., a matrix
$(A\vert B)$, with $A,B\in\M(3,2)$, of rank $3$. In what follows, we always
identify $(A\vert B)$ with its row space.

Let
\begin{equation*}
\cS:=\{(X\vert\1)\mid X\in\bS\}.
\end{equation*}
Then of course each $3\times 6$ matrix in $\cS$ has rank $3$ and thus will be
considered as a plane of $\PG(5,2)$. The bijection $\bS\to \cS: X\mapsto
(X\vert\1)$ gives us another model of $\GQ(2,4)$, where the points are planes
in $\PG(5,2)$. We call this model $\GQ(\cS)$.
Note that for $X\in \bS$  the coordinate vector $\alpha(X)$ defined in  (\ref{alpha}) consists exactly of the six Pl\"ucker coordinates (i.e.\ the $3\times 3$-minors) of $(X\vert\1)$ that have multiplicity one.

We study the planes of $\cS$ more in detail. It is shown in \cite{wan} (see
also \cite{bhavl}, \cite{bh}) that all these planes are totally isotropic
w.r.t.\ a symplectic polarity of $\PG(5,2)$. Moreover, they are all skew to the
plane $(\1\vert 0)$ and to the plane $(0\vert\1)$, as two planes $(A\vert B)$
and $(C\vert D)$ are skew iff the $6\times 6$ matrix
$\begin{pmatrix}A&B\\C&D\end{pmatrix}$ is invertible.
For two planes $(X\vert \1)$ and $(Y\vert \1)$ (with $X,Y\in\M(3,2)$ arbitrary)
we have \beq\label{rank} \rk(X-Y) =k\iff \dim ((X\vert \1)\cap (Y\vert \1))=3-k
\eeq (see \cite[Prop.\ (3.30)]{wan}; here $\dim$ means the vector space
dimension). Note that in (\ref{rank}) we might also write $X+Y$ instead of
$X-Y$ since the ground field is $F=\GF(2)$.
Now we study the following subsets of $\cS$:
\begin{equation*}\cD:=\{(X\vert\1)\mid X\in\bD\},\, \cU:=\{(X\vert\1)\mid X\in\bU\},\,
\cV:=\{(X\vert\1)\mid X\in\bV\}.
\end{equation*}
\begin{rem}\label{spreads}
The sets $\cU':=\{(\1\vert 0),(0\vert\1), (\1\vert\1)\}\cup\cU$ and
$\cV':=\{(\1\vert 0),(0\vert\1), (\1\vert\1)\}\cup\cV$ are \emph{spreads}, i.\,e.
partitions of the point set of $\PG(5,2)$ into planes.
\end{rem}
\begin{proof}
By (\ref{rank}), any two different planes of $\cU'$ (or $\cV'$, respectively)
are skew. Since $\PG(5,2)$ altogether has $63$ points and each plane has $7$
points, the assertion follows.
\end{proof}
It is easy to see that the intersection of the planes $(X\vert \1)$ and
$(\1\vert \1)$ consists exactly of the points $F(x_1,x_2,x_3,x_1,x_2,x_3)$
where $x=(x_1,x_2,x_3)$ is an eigenvector of $X$. (Note that we let the
matrices act from the right, so the eigenvectors are rows and not columns.)
This implies:
\begin{rem}\label{meet}
For each $D_i\in\bD$ the plane $(D_i\vert\1)\in\cD$ meets the distinguished
plane $ (\1\vert \1)$ either in a point or in a line (and this latter case
appears exactly if $i=1,2,3$).
\end{rem}
By Rem.\ \ref{spreads} the plane $(X\vert\1)\in\cU\cup\cV$ is always skew to $
(\1\vert \1)$.
Next we want to find out how the planes $(U\vert \1)\in\cU$ intersect the planes
$(X\vert\1)\in\cD\cup\cV$. For this we need a lemma on an action of the cyclic
group $G=\{\1\}\cup\bU$:
\begin{lem}\label{lemmagroup}
The mapping $\rho: U\mapsto \rho_U:X\mapsto UXU$ ($U\in G$) is an action of $G$
on the $21$-element set $\bD\cup \bV$. There are three orbits of this action,
each orbit contains $2$ elements of $\bV$ and $5$ elements of $\bD$, among
these exactly one of the involutions $D_1, D_2, D_3$. (So $D_1, D_2, D_3$ can
be chosen as representatives of the orbits.)
The same holds for the action of $H=\{\1\}\cup\bV$ on $\bD\cup \bU$.
\end{lem}
\begin{proof}
For $U\in G$ and $X\in \bD\cup \bV$, the matrix $UXU$ is again symmetric and of
course again invertible, so $UXU\in\bS\cup\{\1\}$. In addition, since $G$ is a
group and $X$ does not belong to $G$, we are sure that $UXU$ belongs to
$\bD\cup \bV$. As $G$ is commutative, we have $\rho_{UU'}=\rho_U\rho_{U'}$, so
$\rho$ is a group action. The statement on the orbits can be checked by a
direct computation.
\end{proof}
Now we extend the action of $G$ on $\bD\cup \bV$ to the plane set $\cD\cup
\cV$. Consider again $U\in G$. Then the $6\times 6$ matrix $\begin{pmatrix}
U&0\\0&U^{-1}\end{pmatrix}$ induces a collineation of $\PG(5,2)$. This
collineation maps planes to planes; in particular, for each matrix $X$ we have
\begin{equation*}
(X\vert\1)\mapsto (X\vert\1)\begin{pmatrix} U&0\\0&U^{-1}\end{pmatrix}= (XU\vert U^{-1})=(UXU\vert\1).
\end{equation*}
Moreover, we have that $(X\vert\1)$ meets $(\1\vert\1)$ in a line, if, and only
if, $(UXU\vert\1)$ meets $(U\1U\vert\1)=(U^2\vert \1)$ in a line.

Consider now an arbitrary $X\in \bD\cup\bV$. Then the orbit of $X$ under the
action of $G$ contains exactly one of the matrices $D_1,D_2,D_3$, say $D_1$,
i.\,e., $X=UD_1U$ for some $U\in G$. By Rem.~\ref{meet}, $(D_1\vert\1)$ meets
$(\1\vert\1)$ in a line. So $(X\vert\1)=(UD_1U\vert\1)$ meets $(U^2\vert\1)$ in
a line. Recall now that the set $\cU'$ defined in Rem.\ \ref{spreads} is a
spread of $\PG(5,2)$, so each of the $4$ remaining points of the plane
$(X\vert\1)$ (not in $(U^2\vert\1)$) must be contained in exactly one of the
planes $(Y\vert\1)$ with $Y\in G\setminus\{U^2\}$ (recall that $(X\vert\1)$ is
skew to $(\1\vert 0)$ and $(0\vert\1)$). Since no $3$ of the $4$ points are
collinear, we have that $4$ of these planes are needed, i.\,e., $4$ of these
planes meet $(X\vert\1)$ in exactly one point. By Rem.\ \ref{meet}, the plane
$(\1\vert\1)$ is among these $4$ planes if, and only if, $X=D_1$.
Altogether, we have the following result:
\begin{prop}\label{1.1}
Let $X\in\bD\cup\bV$. Then the plane $(X\vert \1)\in\cD\cup\cV$ meets
\begin{enumerate}
\item[\rm(a)] exactly $4$ planes of $\cU$ in a point and no plane of $\cU$ in
    a line, if $X\in\{D_1,D_2,D_3\}$,
\item[\rm(b)] exactly $3$ planes of $\cU$ in a point and exactly one plane of
    $\cU$ in a line, if $X\in\bD\setminus\{D_1,D_2,D_3\}$,
\item[\rm(c)] exactly $4$ planes of $\cU$ in a point and exactly one plane of
    $\cU$ in a line, if $X\in\bV$.
\end{enumerate}
Consequently, $(X\vert \1)$ is skew to exactly two planes of $\cU$, if
$X\in\bD$, and $(X\vert \1)$ is skew to exactly one plane of $\cU$, if
$X\in\bV$.
The same holds if the roles of $\bU$ and $\bV$ are interchanged.
\end{prop}
By the above, for each $X\in\bU$ there is a unique in $Y\in\bV$ such that
$(X\vert \1)$ is skew to $(Y\vert \1)$ (and vice versa). We write $Y=X'$ (and
$X=Y'$). A direct computation using (\ref{rank}) yields that $U_i'=V_i$.

Now we study collinearity in $\GQ(\cS)$. We write  $(X\vert \1)\sim (Y\vert \1)$ iff
the planes $(X\vert \1), (Y\vert \1)\in\cS$ are collinear. By definition of $\GQ(\cS)$ this is equivalent to $X\sim Y$. Using Propositions \ref{collinear} and
\ref{1.1}, we obtain the following:
\begin{prop}\label{collPlanes}
Let $(X\vert \1), (Y\vert \1)\in\cS$. Then the following holds:
\begin{enumerate}
\item[\rm(a)] If $X,Y\in\bU\cup \bV$, then $(X\vert \1)\sim (Y\vert \1)$ iff
    the two planes meet. This is the case exactly if $X=Y$ or $X\in\bU$,
    $Y\in\bV\setminus\{X'\}$ or $X\in\bV$, $Y\in\bU\setminus\{X'\}$.
\item[\rm(b)] If $X\in\bU\cup \bV, Y\in\bD$, then $(X\vert \1)\sim (Y\vert
    \1)$ iff the two planes are skew. For each $Y\in\bD$ there are exactly
    two matrices $X_1,X_2\in\bU$ and two matrices $X_3,X_4\in\bV$
    satisfying this condition; moreover, w.l.o.g., $X_3=X_1'$ and
    $X_4=X_2'$.
\item[\rm(c)] If $X,Y\in\bD$, then $(X\vert \1)\sim (Y\vert \1)$ iff the two
    planes meet. For each $X\in\bD$ there are exactly $6$ matrices
    $Y\in\bD$, $Y\ne X$, satisfying this condition.
\end{enumerate}
\end{prop}
\begin{proof}
(a): By \ref{collinear} (a) we have that $(X\vert \1)\sim (Y\vert \1)$ iff $0=\det
(X+Y)$, which by (\ref{rank}) is equivalent to the statement that the planes
meet. The second statement follows from \ref{1.1}.

(b): By \ref{collinear} (c)
we have that $(X\vert \1)\sim (Y\vert \1)$ iff $1=\det (X+Y)$, which by
(\ref{rank}) is equivalent to the statement that the planes are skew. The
second statement follows from \ref{1.1}. The statement that $X_3=X_1'$
and $X_4=X_2'$ can be checked by an explicit computation. (See also Prop.\ \ref{iso}.)

(c): The first statement can be shown as in (a). The second statement follows
from (a) and (b) and the fact that each point $p$ of $\GQ(2,4)$ is collinear to
exactly $10$ points different from~$p$.
\end{proof}

The statements from above suggest how to find an explicit isomorphism from
$\GQ(2,4)$ onto $\GQ(\bS)$ (or $\GQ(\cS)$) such that $\bD$ corresponds to the
doily and $\bU,\bV$ correspond to the double-six.
\begin{prop}\label{iso}
The following is an isomorphism from $\GQ(\bS)$ to $\GQ(2,4)$:
\begin{equation*} \forall i\in\{1,2,\ldots,6\}: U_i\mapsto i,\; V_i\mapsto i',
\end{equation*}
\begin{equation*} D_1\mapsto \{3,5\},\; D_2\mapsto \{1,4\}, \; D_3\mapsto \{2,6\}, \; D_4\mapsto \{1,2\}, \; D_5\mapsto \{4,5\},\end{equation*}
\begin{equation*} D_6\mapsto \{1,6\}, \; D_7\mapsto \{3,4\}, \; D_8\mapsto \{3,6\}, \; D_9\mapsto \{2,5\}, \; D_{10}\mapsto \{4,6\}, \end{equation*}
\begin{equation*} D_{11}\mapsto \{1,3\}, \; D_{12}\mapsto \{1,5\}, \; D_{13}\mapsto \{2,4\}, \; D_{14}\mapsto \{2,3\}, \; D_{15}\mapsto \{5,6\}
\end{equation*}
\end{prop}
\begin{proof}
Explicit computation using \ref{collinear}.
\end{proof}

We mention still another model:
The mapping $\pi:X\mapsto X+\1$ from above can be transferred to the planes
$(X\vert\1)$ via $(X\vert\1)\mapsto (X+\1\vert\1)$. This is induced by a
collineation of $\PG(5,2)$, namely, the one given by the $6\times 6$ matrix
$\begin{pmatrix} \1&\1\\0&\1\end{pmatrix}$. We call the image of $\cS$ under
this collineation $\pi(\cS)$. Then it is clear that $\pi(\cS)$ is the image of
$\pi(\bS)=\cQ$ under the representation $X\mapsto (X\vert\1)$, and $\pi(\cS)$
can be seen as the point set of another model of $\GQ(2,4)$. As in $\GQ(\cQ)$,
we then have that $(X\vert \1),(Y\vert\1)\in\pi(\cS)$ are collinear iff
$0=\langle X\vert Y\rangle=\langle X\vert Y\rangle_0=\det(X+Y)+\det X+\det Y$.
The point set $\pi(\cS)$ consists exactly of those planes $(X\vert \1)$, $X\in
J=\s(3,2)$, that are skew to $(\1\vert\1)$ (because by (\ref{qdet2}) we have
$0=q(X)=\det(X+\1)+1$) and different from $(0\vert\1)$.

\section{An interesting physical application}

The ideas discussed above can also be relevant from a physical point of view. This is because the geometry of the generalized quadrangle GQ$(2,4)$ reproduces completely the properties of the $E_{6(6)}$-symmetric entropy formula describing black holes and black strings in $D=5$ supergravity \cite{levayetal}. As a detailed discussion of this issue lies far beyond the scope of the present paper we just outline the gist, referring the interested reader to \cite{levayetal} for more details and further literature. The 27 black hole/string charges correspond to the points and the 45 terms in the entropy formula to the lines of GQ$(2,4)$. Different truncations with 15, 11 and 9 charges are represented by three distinguished subconfigurations of GQ$(2,4)$, namely by a copy of GQ$(2,2)$, the set of points collinear to a given point, and a copy of GQ$(2,1)$, respectively. In order to obtain also the correct {\it signs} for the terms in the entropy formula, it was necessary to employ a {\it non}-commutative labelling for the points of GQ$(2,4)$. This was furnished by certain elements of the real three-qubit Pauli group \cite{levayetal};  now it is obvious that the set of invertible matrices from $J$ = $\s(3,2)$ lends itself as another candidate to do such job. 	A link between the two labellings can, for example, be established by employing the plane representation of $\bS=J^*\setminus\{\1\}$, transforming each matrix via associated Pl\"ucker coordinates into a binary six-vector (eq.\,(\ref{alpha})) 
and encoding the latter in a particular way into the tensor product of a triple of the real Pauli matrices (see \cite{vl} for motivations and further details of the final step of such correspondence in a more general physical setting).


\bigskip

{\bf Acknowledgements.} The work on this topic began in the framework of the ZiF Cooperation Group ``Finite Projective Ring Geometries: An Intriguing Emerging Link Between Quantum Information Theory, Black-Hole Physics, and Chemistry of Coupling," held in 2009 at the Center for Interdisciplinary Research (ZiF), University of Bielefeld, Germany. It was also partially supported by the VEGA grant agency, projects 2/0092/09 and 2/0098/10. We thank Prof. Hans Havlicek (Vienna) for several clarifying comments and Dr. Petr Pracna (Prague) for an electronic version of the figure.


\end{document}